\documentstyle[10pt,amscd,xypic,amssymb]{amsart}

\xyoption{all}
\CompileMatrices

\newcommand{\nc}{\newcommand}

\makeatletter
\@addtoreset{equation}{section}
\makeatother

\newenvironment{proof}{{\noindent \textbf{Proof}\,\,}}{\hspace*{\fill}$\Box$\medskip}

\newtheorem{theorem}[equation]{Theorem}
\newtheorem{proposition}[equation]{Proposition}

\newtheorem{corollary}[equation]{Corollary}

\theoremstyle{definition}

\theoremstyle{remark}

\nc{\fa}{{\mathfrak{a}}}
\nc{\fb}{{\mathfrak{b}}}
\nc{\fg}{{\mathfrak{g}}}
\nc{\fh}{{\mathfrak{h}}}
\nc{\fj}{{\mathfrak{j}}}
\nc{\fn}{{\mathfrak{n}}}
\nc{\fu}{{\mathfrak{u}}}
\nc{\fp}{{\mathfrak{p}}}
\nc{\fr}{{\mathfrak{r}}}
\nc{\ft}{{\mathfrak{t}}}
\nc{\fsl}{{\mathfrak{sl}}}
\nc{\fgl}{{\mathfrak{gl}}}
\nc{\hsl}{{\widehat{\mathfrak{sl}}}}
\nc{\hgl}{{\widehat{\mathfrak{gl}}}}
\nc{\hg}{{\widehat{\mathfrak{g}}}}
\nc{\chg}{{\widehat{\mathfrak{g}}}{}^\vee}
\nc{\hn}{{\widehat{\mathfrak{n}}}}
\nc{\chn}{{\widehat{\mathfrak{n}}}{}^\vee}

\nc{\fC}{{\mathfrak{C}}}
\nc{\fZ}{{\mathfrak{Z}}}

\nc{\pol}{{\text{Poles}}}

\nc{\BA}{{\mathbb{A}}}
\nc{\BC}{{\mathbb{C}}}
\nc{\BM}{{\mathbb{M}}}
\nc{\BN}{{\mathbb{N}}}
\nc{\BQ}{{\mathbb{Q}}}
\nc{\BF}{{\mathbb{F}}}
\nc{\BP}{{\mathbb{P}}}
\nc{\BR}{{\mathbb{R}}}
\nc{\BZ}{{\mathbb{Z}}}

\nc{\CA}{{\mathcal{A}}}
\nc{\CB}{{\mathcal{B}}}
\nc{\CE}{{\mathcal{E}}}
\nc{\CF}{{\mathcal{F}}}
\nc{\tCF}{{\widetilde{\CF}}}
\nc{\oCF}{{\overline{\CF}}}
\nc{\CG}{{\mathcal{G}}}
\nc{\CL}{{\mathcal{L}}}
\nc{\CM}{{\mathcal{M}}}
\nc{\CH}{{\mathcal{H}}}
\nc{\CN}{{\mathcal{N}}}
\nc{\CO}{{\mathcal{O}}}
\nc{\CP}{{\mathcal{P}}}
\nc{\CQ}{{\mathcal{Q}}}
\nc{\CS}{{\mathcal{S}}}
\nc{\CT}{{\mathcal{T}}}
\nc{\CU}{{\mathcal{U}}}
\nc{\CV}{{\mathcal{V}}}
\nc{\CW}{{\mathcal{W}}}
\nc{\tCW}{{\widetilde{\CW}}}
\nc{\oCW}{{\overline{\CW}}}

\nc{\uu}{{U_q(\hgl_n)}}
\nc{\uup}{{U_q^+(\hgl_n)}}
\nc{\uug}{{U_q^\geq(\hgl_n)}}
\nc{\uus}{{U_q(\hsl_n)}}
\nc{\uuu}{{U_q(\hgl_1)}}
\nc{\nn}{{\mathbb{N}}^n}

\nc{\tT}{{T}}

\nc{\wfZ}{{\widetilde{\fZ}}}

\nc{\od}{{\overline{d}}}
\nc{\rg}{{\textrm{R}\Gamma}}
\nc{\erg}{{\emph{R}\Gamma}}
\nc{\id}{{\textrm{id}}}

\def\e{\varepsilon}

\def\slz{SL_2(\BZ)}

\def\sym{\text{Sym}}

\def\D{D}

\def\sym{\text{Sym}}

\begin{document}

\title[Operators on Symmetric Polynomials]{\Large{\textbf{Operators on Symmetric Polynomials}}}

\author[Andrei Negut]{Andrei Negut}
\address{Columbia University, Department of Mathematics, New York, USA}
\address{Simion Stoilow Institute of Mathematics, Bucharest, Romania}
\email{andrei.negut@@gmail.com}

\maketitle

%\begin{abstract}

%\end{abstract}

\section{Introduction}

The purpose of this note is to give a brief survey of certain operators on symmetric polynomials. There are many ways to write them, but the way we will mostly focus on was defined in \cite{F} in terms of the shuffle algebra. This algebra was introduced quite generally in \cite{FO}, though we will only need a particular case of their construction, which we will recall in Subsection \ref{sub:shuf}. The shuffle algebra (more precisely, its Drinfeld double) acts on the well-known ring of symmetric polynomials:
$$
\Lambda = \BQ(q,t)[x_1,x_2,...]^{\sym}
$$
This action includes certain well-known operators on this vector space, such as: \\

\begin{itemize} 

\item the operator of multiplication by $f$, for any $f\in \Lambda$, \\

\item the Macdonald eigenoperator $\D_g$ for any $g\in \Lambda$, given by: 

\end{itemize}

\begin{equation}
\label{eqn:eigen}
\D_g(P_\lambda) = g \left[ (1-q)(1-t) \cdot \text{weights outside the partition }\lambda \right] \cdot P_\lambda
\end{equation}
$$$$
\footnote{The above means the following: the Young diagram of the partition $\lambda$ is a finite subset lodged in the corner of the first quadrant. The weights outside the diagram are $q^it^j$ for $i\geq 0, j \geq \lambda_i$. Multiplying them by $(1-q)(1-t)$ should be understood as a plethystic substitution} Throughout this note, $P_\lambda$ will denote the modified Macdonald polynomials, which are orthogonal with respect to the inner product \eqref{eqn:inner}. One of the new formulas we present in this paper expresses $D_g$ as a certain vertex operator in the spirit of \cite{F}, for a wide class of symmetric polynomials $g$ called rim-hook (or ribbon) skew Schur functions. This uses the results of \cite{N} about the structure of the shuffle algebra. In particular, we have: \\

\begin{theorem} 
\label{thm:one}

The Macdonald eigenoperator $\D_n = \D_{p_n}$ is given by: 
$$
\D_n = \int_{|z_1|\gg...\gg |z_n|} \frac {dz_1...dz_n}{(2\pi i)^n z_1... z_n} \cdot \frac { \left( \sum_{i=1}^n \frac {z_n (qt)^{n-i}}{z_i} \right) \prod_{i<j} \omega \left( \frac {z_i}{z_j} \right)}{\left(1 - \frac {z_2qt}{z_1}\right)...\left(1 - \frac {z_{n}qt}{z_{n-1}}\right)}
$$

$$
\exp \left(- \sum_{k\geq 1} \beta_k p_k \frac {(z_1^{-k}+...+z_n^{-k}) q^{-\frac k2} t^{-\frac k2}}k \right)  \exp \left(- \sum_{k\geq 1} \beta_k p_k^\dagger \frac {z_1^k+...+z_n^k}k \right) 
$$
where $\beta_k$ and $\omega(x)$ are defined in \eqref{eqn:inner} and \eqref{eqn:omega}. \\

\end{theorem}

Most of the results presented in this paper have been proved in other works, either explicitly or implicitly. We will therefore not seek to give a self-contained presentation here, but instead provide a survey of the results, with indications of where one can find proofs. We should warn the reader that our notations differ slightly from those of \cite{F}, in that our Macdonald polynomials are obtained from theirs via a plethystic substitution, and our $(q,t)$ are their $(q^{-1},t)$. \\

I would like to thank Boris Feigin, Eugene Gorsky, Andrei Okounkov, Alexander Tsymbaliuk for their patience in explaining many of these things to me, and for helping me to patch up this beautiful picture. I would also like to thank Adriano Garsia and Vadim Gorin for their interest in these operators, which ultimately led me to write this survey note down, and for their useful comments and corrections. 

\section{The operators}

\subsection{} 
\label{sub:algebra} 

We will begin by recalling the double elliptic Hall algebra $\CA$, in the combinatorial presentation of \cite{BS} (see also \cite{SV}). This algebra is generated by elements $u_{m,n}$ for $(m,n)\in \BZ^2 \backslash (0,0)$, subject to the relations:
\begin{equation}
\label{eqn:rel1}
[u_{km,kn}, u_{lm,ln}] = \frac {k \delta_{k+l}^0 \left( (qt)^{\frac {km}2} - (qt)^{-\frac {km}2}\right)}{(q^{\frac k2}-q^{-\frac k2})(t^{\frac k2}-t^{-\frac k2})( (qt)^{\frac k2} - (qt)^{-\frac k2} )} 
\end{equation}
for any coprime $m,n$, and:
\begin{equation}
\label{eqn:rel2}
[u_{m,n}, u_{m',n'}] =  \frac {(qt)^{\frac {m-M}2} \cdot v_{m+m',n+n'}}{(q^{\frac 12}-q^{-\frac 12})(t^{\frac 12}-t^{-\frac 12})( (qt)^{\frac 12} - (qt)^{-\frac 12})}
\end{equation}
for any clockwise oriented triangle $T = \{(0,0), (m,n), (m+m',n+n')\}$ with no lattice points inside and on at least one of the edges, where we write:
$$
1 + \sum_{k\geq 1} v_{km,kn} z^k = \exp \left( \sum_{k\geq 1} (q^{\frac k2}-q^{-\frac k2})(t^{\frac k2}-t^{-\frac k2})((qt)^{\frac k2} - (qt)^{-\frac k2}) \cdot u_{km,kn} \frac {z^k}k \right) 
$$
for any coprime $m,n$, and let $(M,N)$ denote the middle vertex of the triangle $T$ with respect to the horizontal direction. In case two vertices have the same horizontal coordinate, we consider the "rightmost" one to be the one which is higher up. Note that the algebra studied in \cite{BS}, \cite{SV} depended on two central elements $c_1, c_2$, which we specialize to $c_1=1$ and $c_2 = (qt)^{\frac 12}$ throughout this paper. \\

\subsection{} The elliptic Hall algebra $\CA$ is also known as $U_{c_2}(\hgl_1)$, and as such it has a lot in common with the theory of quantum groups. In particular, it can be divided in a positive part, a Cartan part and a negative part, but the interesting thing is that this can be done in many ways. Namely, take any line in the lattice $\BZ^2$ passing through the origin. Depending on whether $(m,n)$ is on one side of this line, on the line, or on the other side of the line, we may call the generator $u_{m,n}$ positive, Cartan, or negative. Then $\CA$ turns out to be a double of its positive half, as in the theory of quasi-triangular Hopf algebras. \\

The ambiguity of the chosen line is no surprise. In fact, there is an almost-action of $\slz$ by automorphisms of the algebra $\CA$, simply by permuting the lattice points $(m,n)$ and the corresponding generators. The word almost is a way to sweep under the rug the fact that it's actually the universal cover of $\slz$ that acts, and it multiplies the generators $u_{m,n}$ by a certain power of the central elements $c_1$ and $c_2$ beside permuting the point $(m,n)$ (see \cite{BS}). If one specializes the central elements as we did, this action is lost. But modulo this technical point, the essential thing is that the algebra $\CA$ is naturally $\slz$ symmetric. \\

\subsection{} 
\label{sub:act1}

From now on, let $\dagger$ denote adjoint with respect to the scalar product: 
\begin{equation}
\label{eqn:inner}
\langle p_k, p_k \rangle = - \frac {k}{\beta_k}, \quad \text{ where }\quad \beta_k = (q^{\frac k2} - q^{-\frac k2})(t^{\frac k2} - t^{-\frac k2})
\end{equation}
\footnote{Note that this is a slight plethystic modification of the usual Macdonald inner product, and the consequence of this will be replacing Macdonald polynomials by their modified versions. In fact, comparing with \cite{F}, the corresponding plethystic substitution is that: 
$$
\text{our }p_k = \text{their }p_k \cdot \frac {t^{-\frac k2}}{q^{-\frac k2}-q^{\frac k2}}
$$ 
and also our $q$ equals their $q^{-1}$, while $t$ is the same} We will also denote by $p_k$ the operator on $\Lambda$ of multiplication with the power-sum function, and hope that this will not cause any ambiguity. In \cite{F}, the authors define an action of $\CA$ on $\Lambda$ by sending $u_{m,n}$ to the operators $U_{m,n} \in \text{End}(\Lambda)$ given by:

$$
U_{m,0} = p_m, \qquad U_{-m,0} = p_m^\dagger, \qquad \qquad \qquad \forall \ m > 0
$$

$$
U_{m,1} = \frac {(qt)^{\frac m2\delta_{m<0}}}{\beta_1} \int z^m \cdot \exp \left(- \sum_{k\geq 1} \beta_k p_k \frac {z^{-k} q^{-\frac k2} t^{-\frac k2}}k \right)  \exp \left(- \sum_{k\geq 1} \beta_k p_k^\dagger \frac {z^k}k \right) Dz
$$

$$
U_{-m,-1} = \frac {(qt)^{-\frac m2\delta_{m < 0}}}{\beta_1} \int z^m \cdot \exp \left(\sum_{k\geq 1} \beta_k p_k \frac {z^{k}}k \right)  \exp \left(\sum_{k\geq 1} \beta_k p_k^\dagger \frac {z^{-k}q^{\frac k2}t^{\frac k2}}k \right) Dz 
$$
for all $m\in \BZ$, where $Dz = \frac {dz}{2\pi i z}$. \footnote{The factors in front of $U_{\pm m,\pm 1}$ and the various powers of $(qt)^{\frac 12}$ are a matter of normalization, and have been chosen so that $U_{m,n}^\dagger = U_{-m,n}$ for all $m,n \in \BZ$} The integrals in the above two formulas are taken over small contours around $\infty$ and $0$, respectively, so it is implicit that the $z$ variable is thought to be much larger (if the sign is $+$) or smaller (if the sign is $-$) than the $x$ variables which determine the symmetric functions $p_k$. Note that the above is enough to define operators $U_{m,n}$ on $\Lambda$, via the relations of Subsection \ref{sub:algebra}. As proved in Proposition 2.9 of \cite{AFS}, these operators verify the relations of $\CA$, and thus give a well-defined action on $\Lambda$. \\

\subsection{} 
\label{sub:act2}

As before, we write $D_n = D_{p_n}$ and let $D_{-n} = D_{p_n}|_{q,t\leftrightarrow q^{-1},t^{-1}}$. The second action of $\CA$ on $\Lambda$ that is described in \cite{AFS}, \cite{SV} is given by sending $u_{m,n}$ to the operators $U_{m,n} \in \text{End}(\Lambda)$ given by:
$$
U_{0,\pm n} = \frac {\D_{\pm n}}{\beta_n}  \quad \forall \ n > 0, \qquad \qquad \qquad U_{\pm 1,0} = X_{\pm 1}, 
$$
Same as above, these elements generate the whole $\CA$, and it is proved in Proposition 2.7 of \cite{AFS} that they induce a well-defined action of $\CA$ on $\Lambda$. \\

\begin{proposition} The two actions of $\CA$ on $\Lambda$ defined in the previous and current subsections are actually one and the same. \\
\end{proposition}

\begin{proof} It is easy to see from \eqref{eqn:rel2} that the algebra $\CA$ is generated by $U_{\pm 1, 0}$ and $U_{0,\pm 1}$. Then one needs to check that these elements act on $\Lambda$ in the same way in the two descriptions. This is obvious for $U_{\pm 1, 0}$, whereas for $U_{0,\pm 1}$ this follows from Theorem 1.2 of \cite{GHT} (after certain plethystic substitutions to unify the conventions). 

\end{proof}

\subsection{} 
\label{sub:shuf}

The above proposition will allow us to obtain integral formulas (in the spirit of Subsection \ref{sub:act1}) for the operators $D_n$. We will do this more generally, by providing integral formulas for the action of all $U_{m,n}$ with $n \neq 0$. This comes about via the shuffle algebra of \cite{FO}, which we now recall. Consider the vector space:
$$
\CV = \bigoplus_{n\geq 0} \BC(q,t)(z_1,...,z_n)^{\sym}
$$
endowed with the \textbf{shuffle product}:
$$
P(z_1,...,z_n) * P'(z_1,...,z_{n'}) = \sym \left[ \frac {P(z_1,...,z_n)P'(z_{n+1},...,z_{n+n'})}{n! n'!}\prod_{n+1 \leq j \leq n+n'}^{1\leq i \leq n} \omega \left(\frac {z_i}{z_j} \right)\right]
$$
where: 
\begin{equation}
\label{eqn:omega}
\omega(x) = \frac {(x - 1)(x - qt)}{(x - q)(x - t)}
\end{equation}
and $\sym$ denotes symmetrization with respect to all variables. Inside the algebra $\CV$, we define the \textbf{shuffle algebra} $\CS$ to consist of symmetric rational functions of the form:
\begin{equation}
\label{eqn:shuf}
P(z_1,...,z_n) =  \frac {p(z_1,...,z_n) \cdot \prod_{1\leq i<j \leq n} (z_i - z_j)^2}{\prod_{1\leq i \neq j \leq n} (z_i - q z_j)(z_i - t z_j)}
\end{equation}
where $p$ is a symmetric Laurent polynomial that satisfies the wheel conditions. \footnote{Although we will not need them explicitly, these conditions are :
$$
p(z_1,z_2,z_3,z_4,...,z_n) = 0 \textrm{ whenever } \left\{ \frac {z_1}{z_2} , \frac {z_2}{z_3}, \frac {z_3}{z_1} \right\} =\left \{q,t,\frac 1{qt} \right\} \text{ as sets}
$$}

\subsection{} 
\label{sub:shufact}

It was proved in \cite{N} that the shuffle algebra $\CS$ is generated by degree one elements $z_1^m$ for $m\in \BZ$. This implies that any shuffle element can be written as a linear combination of:
\begin{equation}
\label{eqn:basis}
P(z_1,...,z_n) = z_1^{m_1} * ... * z_1^{m_n} = \sym \left[z_1^{m_1}...z_n^{m_n} \prod_{1\leq i < j \leq n}  \omega \left(\frac {z_i}{z_j} \right) \right]
\end{equation}
As shown in \cite{SV}, the shuffle algebra is isomorphic to half of the elliptic Hall algebra $\CA$, under either of the following maps:
\begin{equation}
\label{eqn:iso1}
z_1^m \longrightarrow u_{m,1} \cdot (qt)^{-\frac m2\delta_{m<0}}
\end{equation}
for the upper half plane, or:
\begin{equation}
\label{eqn:iso2}
z_1^m \longrightarrow u_{-m,-1} \cdot (qt)^{\frac m2\delta_{m < 0}}
\end{equation}
for the lower half plane. Composing this with the action of $\CA$ on symmetric functions given in Subsection \ref{sub:act1}, we may ask how shuffle elements act on symmetric functions via wither of the above two actions. The simple answer one obtains by iterating those actions is that the shuffle element \eqref{eqn:basis} acts on $\Lambda$ by the formulas:
\begin{equation}
\label{eqn:upp}
U_P^+ = \frac {1}{\beta^n_1} \int_{|z_1|\gg...\gg |z_n|} Dz_1...Dz_n \cdot z_1^{m_1}...z_n^{m_n} \prod_{1\leq i < j \leq n}  \omega \left(\frac {z_i}{z_j} \right) 
\end{equation}

$$
\exp \left(- \sum_{k\geq 1} \beta_k p_k \frac {(z_1^{-k}+...+z_n^{-k}) q^{-\frac k2} t^{-\frac k2}}k \right)  \exp \left(- \sum_{k\geq 1} \beta_k p_k^\dagger \frac {z_1^k+...+z_n^k}k \right) 
$$
for the upper half \eqref{eqn:iso1}, and:
\begin{equation}
\label{eqn:upm}
U_P^- = \frac {1}{\beta^n_1} \int_{|z_1|\ll...\ll |z_n|} Dz_1...Dz_n \cdot z_1^{m_1}...z_n^{m_n} \prod_{1\leq i < j \leq n}  \omega \left(\frac {z_i}{z_j} \right)   
\end{equation}

$$
\exp \left(\sum_{k\geq 1} \beta_k p_k \frac {z_1^{k}+...+z_n^{k}}k \right)  \exp \left(\sum_{k\geq 1} \beta_k p_k^\dagger \frac {(z_1^{-k}+...+z_n^{-k})q^{\frac k2}t^{\frac k2}}k \right)
$$
for the lower half \eqref{eqn:iso2}.  \\

\subsection{}
\label{sub:order}

The above formulas for $U_P^\pm$ have poles when $z_i = q z_j$ and $z_i = t z_j$ for $i<j$. So if we assume $|q|,|t|<1$ (when the sign is $+$) and $|q|,|t|>1$ (when the sign is $-$), we can move the contours to $|z_1| = ... = |z_n|$. These two assumptions may seem contradictory, but they are not: both integrals can be written as a sum of residues which are formal expressions in the parameters $q,t$. The size of these parameters only determines which residues we are considering. So then, after symmetrizing with respect to all the variables involved, we conclude that:
$$
U_P^+ = \frac {1}{\beta^n_1n!} \int^{|q|,|t|<1}_{|z_1| = ... = |z_n|} Dz_1...Dz_n \cdot \sym \left[z_1^{m_1}...z_n^{m_n} \prod_{1\leq i < j \leq n}  \omega \left(\frac {z_i}{z_j} \right)\right] \cdot E_+ = 
$$

\begin{equation}
\label{eqn:upp2}
= \frac {1}{\beta^n_1n!} \int^{|q|,|t|<1}_{|z_1| = ... = |z_n|} P(z_1,...,z_n) \cdot E_+(z_1,...,z_n;x)   Dz_1...Dz_n
\end{equation}
for the upper half \eqref{eqn:iso1}, and:
$$
U_P^- = \frac {1}{\beta^n_1n!} \int^{|q|,|t|>1}_{|z_1| = ... = |z_n|} Dz_1...Dz_n \cdot \sym \left[ z_1^{m_1}...z_n^{m_n}  \prod_{1\leq i < j \leq n}  \omega \left(\frac {z_i}{z_j} \right) \right] \cdot E_- =
$$

\begin{equation}
\label{eqn:upm2}
= \frac {1}{\beta^n_1n!} \int^{|q|,|t|>1}_{|z_1| = ... = |z_n|}  P(z_1,...,z_n) \cdot E_-(z_1,...,z_n;x) Dz_1...Dz_n
\end{equation}
for the lower half \eqref{eqn:iso2}, where $E_+$ and $E_-$ are the products of exponentials in \eqref{eqn:upp} and \eqref{eqn:upm}. While the presence of the symmetrization in the above formulae might make them a bit more cumbersome, we actually prefer \eqref{eqn:upp2} and \eqref{eqn:upm2} because they show that the operators $U_P^\pm$ are actually well-defined for any $P\in \CS$. This is non-trivial, because there are linear relations between the shuffle elements \eqref{eqn:basis}, and it is not a priori clear that they are respected by the formulas \eqref{eqn:upp} and \eqref{eqn:upm}. But if we rewrite these formulas in the form \eqref{eqn:upp2} and \eqref{eqn:upm2}, the fact that they are well-defined becomes immediate. \\

\subsection{} The main result of \cite{N} is to work out which elements of the shuffle algebra correspond to the various $U_{m,n}$. Let us consider the symmetric rational function:
\begin{equation}
\label{eqn:p}
P_{m,n} = \frac {(q-1)^n(t-1)^n}{(q^{g}-1)(t^g-1)} \sym\left[ p_{m,n}(z_1,...,z_n) \prod_{1\leq i<j \leq n} \omega \left( \frac {z_i}{z_j} \right) \right]
\end{equation}
where for any $n>0$, $m\in \BZ$, $g = \gcd(m,n)$, $a = \frac ng$ we write:
$$
p_{m,n}(z_1,...,z_n) =  \frac {\prod_{i=1}^n z_i^{\left \lfloor \frac {im}n \right \rfloor - \left \lfloor \frac {(i-1)m}n \right \rfloor}\sum_{x=0}^{g-1} (q t)^{x} \frac {z_{a(g-1)+1}...z_{a(g-x)+1}}{{z_{a(g-1)}...z_{a(g-x)}}}}{\left(1 - \frac {z_2 qt}{z_1}\right)...\left(1 - \frac {z_{n}qt}{z_{n-1}}\right)}
$$
The main result of \cite{N}, Theorem 1.1, claims that under the isomorphisms \eqref{eqn:iso1} and \eqref{eqn:iso2} we have:
$$
P_{m,n} \longrightarrow u_{m,n} \cdot (qt)^{-\frac m2\delta_{m<0}}  \qquad \qquad P_{m,n} \longrightarrow u_{-m,-n} \cdot (qt)^{\frac {m}2\delta_{m<0}} 
$$
We thus obtain integral formulas for the operators $U_{\pm m, \pm n} : \Lambda \longrightarrow \Lambda$ via \eqref{eqn:upp2} and \eqref{eqn:upm2}. However, the rational function $p_{m,n}$ does not add any new poles to the integral, because its denominator is canceled by the numerator of the product of $\omega$'s. So we can backtrack the discussion in Subsection \ref{sub:order} and move the contours to obtain the following: \\

\begin{theorem} 
\label{thm:power}

For any $m \in \BZ$ and $n>0$, the operators $U_{\pm m,\pm n}$ are given by:
$$
U_{m,n} = \frac {(qt)^{\frac n2}}{(q^{g}-1)(t^g-1)} \int_{|z_1|\gg...\gg |z_n|} Dz_1...Dz_n \cdot p_{m,n}(z_1,...,z_n) \prod_{1\leq i<j \leq n} \omega \left( \frac {z_i}{z_j} \right)
$$

$$
\exp \left(- \sum_{k\geq 1} \beta_k p_k \frac {(z_1^{-k}+...+z_n^{-k}) q^{-\frac k2} t^{-\frac k2}}k \right)  \exp \left(- \sum_{k\geq 1} \beta_k p_k^\dagger \frac {z_1^k+...+z_n^k}k \right) 
$$
and:
$$
U_{-m,-n} = \frac {(qt)^{\frac n2}}{(q^{g}-1)(t^g-1)} \int_{|z_1|\ll ... \ll |z_n|} Dz_1...Dz_n \cdot p_{m,n}(z_1,...,z_n) \prod_{1\leq i<j \leq n} \omega \left( \frac {z_i}{z_j} \right)
$$
$$
\exp \left(\sum_{k\geq 1} \beta_k p_k \frac {(z_1^{-k}+...+z_n^{-k}) }k \right)  \exp \left(\sum_{k\geq 1} \beta_k p_k^\dagger \frac {(z_1^k+...+z_n^k)q^{\frac k2} t^{\frac k2}}k \right) 
$$
When $m=0$, this gives precisely Theorem \ref{thm:one}. \\

\end{theorem}

\subsection{} The above formulas for $U_{\pm m, \pm n}$ are particular cases of a more general construction, which we will now explain. For the remainder of this paper, let us assume $\gcd(m,n)=1$. Then the operators $\{U_{\pm km, \pm kn}\}_{k\in {{\mathbb{N}}}}$ form a commutative family, in virtue of relation \eqref{eqn:rel1}. It is well-known (\cite{BS}, \cite{F}, \cite{SV}) that this family is naturally isomorphic to the ring $\Lambda$ itself, via the map:
$$
p_k \longrightarrow \text{suitably chosen constant} \cdot u_{\pm km, \pm kn}
$$
In terms of the shuffle algebra, we obtain for each $\gcd(m,n)$ an embedding:
$$
\Lambda \stackrel{\Upsilon_{m,n}}\longrightarrow \CS, \qquad \qquad p_k \longrightarrow \frac {(q^k-1)(t^k-1)}{(q^{km}-1)(t^{km}-1)} \cdot P_{km,kn}
$$
of \eqref{eqn:p}. We can describe the image under this map of a set of linear generators of $\Lambda$, known as rim-hook (or ribbon) skew Schur functions. To construct these, note that to any sequence $\e \in \{0,1\}^{k-1}$ we may associate a rim-hook skew Young diagram \footnote{Meaning one which doesn't contain any $2\times 2$ boxes} by the following procedure: starting on the horizontal axis, move one box left or one box up, depending on whether the corresponding entry of $\e$ is 0 or 1. Finally shift the whole diagram horizontally so that it rests on the vertical axis, and let $s_\e \in \Lambda_k$ denote the skew Schur function corresponding to that skew diagram. These skew Schur functions are linear generators of $\Lambda$, although there are linear relations between them. Beside the usual relations:
$$
h_k = s_{(0,...,0)}, \qquad e_k = s_{(1,....,1)},
$$
\begin{equation}
\label{eqn:ps}
p_k = s_{(0,...,0)} - s_{(0,0,...,1)} + ... + (-1)^{k-1} s_{(1,...,1)}
\end{equation}
we do not know how other well-known symmetric functions can be written linearly in terms of $s_\e$, although this would be a very interesting question. Then it is proved in \cite{N} that the map $\Upsilon_{m,n}$ sends $s_\e$ to the shuffle element:
$$
(-qt)^{\#\text{ of ones in }\e} \cdot \sym \left[ \frac {\prod_{i=1}^{kn} z_i^{\left \lfloor \frac {im}n \right \rfloor - \left \lfloor \frac {(i-1)m}n \right \rfloor - \e_{\frac ik} + \e_{\frac {i-1}k}}}{\left(1 - \frac {z_2 qt}{z_1}\right)...\left(1 - \frac {z_{kn} qt}{z_{kn-1}}\right)} \prod_{1\leq i<j \leq kn} \omega \left( \frac {z_i}{z_j} \right) \right]
$$
\footnote{We set $\e_x=0$ unless $x\in \{1,...,k-1\}$} This allows us to obtain formulas as in Theorem \ref{thm:power} for any rim-hook skew Schur function in the operators $u_{km,kn}$. In the particular case when $m=0$, we conclude the following: \\

\begin{corollary} The Macdonald eigenoperator $D_\e = D_{s_\e}$ acts on $\Lambda$ by:
$$
D_\e = \int_{|z_1|\gg...\gg |z_n|} Dz_1...Dz_n \frac { \prod_{i=1}^{n-1} \left(-\frac {z_{i+1} qt}{z_i}\right)^{\e_i} \prod_{i<j} \omega \left( \frac {z_i}{z_j} \right)}{\left(1 - \frac {z_2 qt}{z_1}\right)...\left(1 - \frac {z_{n} qt}{z_{n-1}}\right)}
$$

\begin{equation}
\label{eqn:want}
\exp \left(- \sum_{k\geq 1} \beta_k p_k \frac {(z_1^{-k}+...+z_n^{-k}) q^{-\frac k2} t^{-\frac k2}}k \right)  \exp \left(- \sum_{k\geq 1} \beta_k p_k^\dagger \frac {z_1^k+...+z_n^k}k \right) \qquad \qquad \quad
\end{equation}

\end{corollary}

\begin{proof} The Littlewood-Richardson rule easily implies that $s_\e\cdot s_{\e'} = s_{\e0\e'} + s_{\e1\e'}$, hence we can infer a similar relation for the Macdonald eigenoperators:
\begin{equation}
\label{eqn:lr}
D_\e \cdot D_{\e'} = D_{\e0\e'} + D_{\e1\e'}
\end{equation}
If we write $D_\e'$ for the operator in the RHS of \eqref{eqn:want}, it is easy to see that:
$$
D'_\e \cdot D'_{\e'} = D'_{\e0\e'} + D'_{\e1\e'}
$$
and hence the induction hypothesis implies the equality between differences of two $D$s and differences of two $D'$s. In order to infer that all $D$s are equal to the corresponding $D'$s, it is enough to prove so for a single $\e$, and in fact it is enough to claim it for $p_n$. In this case, the desired equality is simply Theorem \ref{thm:one}.

\end{proof}

\end{document}